\documentclass[draft, reqno]{amsart}
\usepackage[all]{xy}                        %

\CompileMatrices                            

\UseTips                                    

\input xypic
\usepackage[bookmarks=true]{hyperref}       

\usepackage{amssymb,latexsym,amsmath,amscd,eqnarray}
\usepackage{xspace}
\usepackage{enumerate}
\usepackage{graphicx}
\usepackage{vmargin}
\usepackage{todonotes}
\usepackage{xcolor}
\usepackage{soul}

\DeclareMathOperator{\red}{red}
\DeclareMathOperator{\Res}{Res}

\reversemarginpar

\theoremstyle{plain}
\newtheorem{theorem}{Theorem}[section]
\newtheorem*{theorem*}{Theorem}
\newtheorem{proposition}[theorem]{Proposition}
\newtheorem{corollary}[theorem]{Corollary}
\newtheorem{lemma}[theorem]{Lemma}

\theoremstyle{definition}
\newtheorem{definition}[theorem]{Definition}

\newtheorem{remark}[theorem]{Remark}
\newtheorem{example}[theorem]{Example}

\newcommand{\enm}[1]{\ensuremath{#1}}          %

\newcommand{\cal}[1]{\mathcal{#1}}

\newcommand{\II}{\enm{\mathbb{I}}}
\newcommand{\NN}{\enm{\mathbb{N}}}

\newcommand{\ZZ}{\enm{\mathbb{Z}}}

\newcommand{\PP}{\enm{\mathbb{P}}}

\newcommand{\VV}{\enm{\mathbb{V}}}

\newcommand{\TT}{\enm{\mathbb{T}}}
\newcommand{\UU}{\enm{\mathbb{U}}}

\newcommand{\Bb}{\enm{\cal{B}}}

\newcommand{\Dd}{\enm{\cal{D}}}

\newcommand{\Ii}{\enm{\cal{I}}}

\newcommand{\Oo}{\enm{\cal{O}}}

\newcommand{\Rr}{\enm{\cal{R}}}

\renewcommand{\phi}{\varphi}
\renewcommand{\theta}{\vartheta}
\renewcommand{\epsilon}{\varepsilon}

\begin{document}

\title{Finite $0$-dimensional multiprojective schemes and their ideals}
\thanks{Last updated: July 13, 2021}

\author[E. Ballico]{Edoardo Ballico}
\address{Dept. of Mathematics\\
 University of Trento\\
38123 Povo (TN), Italy}
\email{ballico@science.unitn.it}

\author[E. Guardo]{Elena Guardo}
\address{Dipartimento di Matematica e Informatica\\
		Viale A. Doria, 6 - 95100 - Catania, Italy}
	\email{guardo@dmi.unict.it} \urladdr{www.dmi.unict.it/~guardo}


\begin{abstract} We study finite $0$-dimensional  schemes in product of multiprojective space and their ideals. In particular, we describe the  set of generators  of the ideal defining a  $0$-dimensional scheme in the case  $\mathbb P^{1}\times\cdots  \times\mathbb P^{1}$  and in the case $\PP^{n_1}\times \cdots \times \PP^{n_k}$ with $n_i\in \{1,2\}$ for all $i$. We also check very ampleness for zero-dimensional schemes of general points in the multiprojective spaces.
\end{abstract}

\maketitle

\section{Introduction}

The study of the homogeneous ideal of projective varieties has given several deep results.  Rather than studying points in a single projective space $\mathbb{P}^n$, one can consider the study of a set of points in a multiprojective space $\mathbb P^{n_1}\times\cdots  \times\mathbb P^{n_r}.$  But the two settings can be significantly different. For example, one of the main differences between points in  $\mathbb{P}^n$, and points in $\mathbb P^{n_1}\times\cdots  \times\mathbb P^{n_r}$  with $r\geq 2$ is that in the former, any zero-dimensional scheme has a Cohen-Macaulay coordinate ring, while in the latter, this is no longer true. Properties of Arithmetically Cohen Macaulay (ACM)  sets of points in $\mathbb P^{1}\times\mathbb P^{1}$ and their combinatorial structure are collected and described in \cite{GVbook2015}.  Natural generalizations are the study  ACM sets of (reduced) points in $\mathbb P^{n}\times\mathbb P^{m}$ and $\mathbb P^{1}\times\cdots  \times\mathbb P^{1}$ that also can be characterized using their combinatorial structures (see for instance \cite{FGM2018, FM1,FM2, GuVT2008, LL, VT1,VT2}).
As we can see from the literature, most of the known results are about the ACM property.  It is connected to  classical problems of  algebraic geometry  related to the dimension of certain secant varieties of  Segre varieties, as in \cite{cco, cgg1} just to cite some of them.


A first step of studying $0$-dimensional  schemes in product of multiprojective spaces can be considered \cite{GMR1992}, where the authors investigated the structure of the Hilbert function of a $0$-dimensional  schemes $X$ in $\mathbb P^{1}\times\mathbb P^{1}$ and studied the relationship between the Hilbert function and the cohomology groups of the ideal sheaf $\mathcal I_X$. It followed also the study of the resolution of $0$-dimensional  schemes which are generic in  $\mathbb P^{1}\times\mathbb P^{1}$ (see for instance \cite{GMR1994, GMR1996}) and the beginning of the study of virtual (scheme-theoretic) complete intersections in $\mathbb{P}^1 \times \mathbb{P}^1$ in \cite{GMZ2013}.  

Inspired by recent results on scheme theoretic, virtual resolutions in a multigraded setting, such as \cite{BES, GLM, HNVT}, we will focus our study in describing some properties of a $0$-dimensional  schemes in product of multiprojective spaces $\PP^{n_1}\times \cdots \times \PP^{n_k}$. 
We aim to generalize some results known for  sets of distinct points in $\PP^{n_1}\times \cdots \times \PP^{n_k}$ and their Hilbert  functions that can be found in \cite{TesiAdam,VT1,VT2, VT3} to  $0$-dimensional  schemes. 
In those cited papers, Van Tuyl generalized a well known classical result about the eventual behaviour of the Hilbert function of a set of points in $\PP^n$ to sets of points in $\PP^{n_1}\times \cdots \times \PP^{n_k}$ introducing the notion of {\it border} of a Hilbert function of a set of distinct points in $\PP^{n_1}\times \cdots \times \PP^{n_k}$. It divides the values of the Hilbert function in  those values which need to be computed and those values which are related to the eventual behaviour of the Hilbert function.  In particular, Van Tuyl found the  Hilbert function $H_Z(a_1,\dots ,a_k)$  of a set of distinct points in $\PP^{n_1}\times \cdots \times \PP^{n_k}$  computing it for only a finite number of $(a_1,\dots ,a_{k})\in \NN^k$ . The other values of the $H_Z(a_1,\dots ,a_k)$ are then easily determined from Corollary 4.7 in \cite{VT1}.
In the literature, multigraded Hilbert functions of reduced  $0$-dimensional scheme  on toric varieties can be found in  \cite{msMultigraded} where the authors generalize some results on the Hilbert function  in \cite{VT1} and made some applications to Coding Theory.
In this paper we are interested in generalizing some of the cited known results on the multigraded Hilbert function from  sets of distinct points to  $0$-dimensional  schemes in product of multiprojective spaces $\PP^{n_1}\times \cdots \times \PP^{n_k}$.
 After introducing some definition and  results, in Sections \ref{generators} and \ref{generators2} we explicity describe the generators  of the ideal defining a  $0$-dimensional scheme in the case  $\mathbb P^{1}\times\cdots  \times\mathbb P^{1}$ (see Theorem \ref{bb1}) and in the case $\PP^{n_1}\times \cdots \times \PP^{n_k}$ with $n_i\in \{1,2\}$ for all $i$ (see Theorem \ref{p2p1}). 

Section \ref{BPfreeness} is devoted  to check very ampleness zero-dimensional schemes of general points in the multiprojective spaces, essentially as tangent vectors. Double points can be considered as
differentials and after Terracini lemma they had a prominent role. However, the original classical motivation (knowing if a
given linear system is base point free or very ample) may be proved even in cases in which the homogeneous ideal (or the
multihomogeneous ideal) is not generated at the degree (or the multidegree) in which we want to prove base-point freeness.

{\bf Acknowledgement} The second author was partially supported by Universit\`{a} degli Studi di Catania, "Progetto Piaceri 2020/2022 Linea di intervento 2". All the authors are members of GNSAGA of INdAM (Italy).




\section{General definitions and results}\label{notation}
In this section we collect a few results and many definitions which make sense for an arbitrary multiprojective space.

We work over an algebraically closed field $K$ of characteristic 0. In the following
we will always deal with a multiprojective space of $k > 0$ factors of the form
$X := \mathbb P^{n_1}\times\cdots  \times\mathbb P^{n_k}$ and we put  $\Rr := K[x_{ij}]$, $i=1,\dots ,k$, $0 \le j\le n_i$ .


If $k>1$ we set $X_i:= \prod _{j\ne i} \PP^{n_j}$ and  we denote by $\eta _i: X\to X_i$
 the projection which forgets the $i$-th coordinate of the points of $X$ and by $\pi_i: X\to \PP^{n_i}$ for $i=1,\dots ,k$ the projection of  X onto the $i$-th factor.  For $i=1,\dots ,k$ let $\epsilon _i$ (resp. $\hat{\epsilon}_i$) denote the multiindex $(a_1,\dots
,a_k)\in \NN^k$ with $a_i=1$ and $a_j=0$ for all $j\ne i$ (resp. $a_i=0$ and $a_j =1$ for all $j\ne i$). 



\begin{remark}\label{dd02}
The K\"{u}nneth formula gives $h^1(\Oo_X(a_1,\dots ,a_k))=0$ for all $(a_1,\dots ,a_k)\in \ZZ^k$ such that $a_i\ge -n_i$ for
all $i$. Thus $h^1(\Oo_X(a_1,\dots ,a_k))=0$ if $a_i\ge -1$ for all $i$.
\end{remark}

\begin{definition}\label{dd01}Let $Z$ be a zero-dimensional scheme of any projective variety $Y$. We will say that $Z$ has {\it
maximal rank} if for all $L\in \mathrm{Pic}(Y)$ either $h^0(\Ii_Z\otimes L)=0$ or $h^1(\Ii _Z\otimes L) =h^1(L)$. 
\end{definition}
\begin{remark}Take $Z$, $Y$ and $L\in \mathrm{Pic}(Y)$ as in Definition \ref{dd01}. A standard
exact sequence shows that $Z$ has maximal rank if and only if for each $L\in \mathrm{Pic}(Y)$ the restriction map $H^0(L)\to
H^0(L_{|Z})$ is either injective or surjective. If $Y$ is a multiprojective space, then $h^1(L)=0$ for all $L\in
\mathrm{Pic}(Y)$ and hence $Z$ has maximal rank if and only  for all $L\in \mathrm{Pic}(Y)$ either $h^0(\Ii_Z\otimes L)=0$ or
$h^1(\Ii _Z\otimes L) =0$.
\end{remark}   
\begin{remark}\label{0bg01}  Set $X := \mathbb P^{n_1}\times\cdots  \times\mathbb P^{n_k}$ and
let $W, Z\subset X$ be zero-dimensional schemes such that $W\subseteq Z$. Obviously $h^1(\Ii_W\otimes L) \le h^1(\Ii_Z\otimes L)$ for all $L\in \mathrm{Pic}(X)$.
\end{remark}

We will need the process of residuation with respect to a divisor. We can state this process
in complete generality for any projective variety $X$. For any zero-dimensional scheme, $Z\subset X$
and for any effective divisor $D$ on $X$ the \lq\lq residue of $Z$ w.r.t. $D$\rq\rq \ is the scheme $Res_{D}(Z)$ defined
by the ideal sheaf $\Ii_Z : \Ii_D$, where $\Ii_Z$ and $\Ii_D$ are the ideal sheaves of $Z$ and $D$, respectively.
The multiplication by local equations of $D$ defines the exact sequence of sheaves:
\begin{equation}\label{residue}
0 \rightarrow \Ii_{Res_{D}(Z)}(-D) \rightarrow \Ii_Z \rightarrow \Ii_{D\cap Z,D}\rightarrow 0
\end{equation}
\noindent where the rightmost sheaf is the ideal sheaf of $D\cap Z$ in $D$.

For any zero-dimensional subscheme of a multiprojective space $X$ let $\ll Z\gg$ denote the minimal multiprojective space
containing $Z$. Note that $\ll Z\gg = \prod_{i=1}^{k} \langle \pi _i(Z)\rangle$, where $\langle \ \ \rangle$ denote the linear
span.

Let $Z\subset X= \PP^{n_1}\times \cdots \times \PP^{n_k}$ be a zero-dimensional scheme and set 

\begin{eqnarray*}
\II(Z):= & \oplus _{(a_1,\dots,a_k)}H^0(\Ii_Z(a_1,\dots ,a_k)) \textrm{  as $\Rr$-modulo}\\  
\II_0(Z):=& \{(a_1,\dots ,a_k)\in \NN^n\mid h^0(\Ii_Z(a_1,\dots ,a_k))>0\} \\
\II_1(Z):=& \{(a_1,\dots ,a_k)\in \NN^k\mid h^1(\Ii_Z(a_1,\dots ,a_k))>0\}.
\end{eqnarray*}

We have that $Z$ has maximal rank if and only if $\II_0(Z)\cap \II_1(Z) =\emptyset$. 

\begin{definition}
For each $(a_1,\dots ,a_k)\in \NN^k$ the set $D(a_1,\dots ,a_k)$ of its {\it descendants} is the set of all $(b_1,\dots ,b_k)\in \NN^k$
such that $b_i\ge a_i$ for all $i$ and $b_i\ne a_i$ for at least one $i$.  
\end{definition}

A descendant $(b_1,\dots ,b_k)$ of $(a_1,\dots
,a_k)$ is an {\it immediate descendant} if $\sum _{i=1}^{k} (b_i-a_i) =1$, i.e. if $(b_1,\dots ,b_k)=(a_1,\dots ,a_k)+\epsilon
_i$ for some $i=1,\dots ,k$. Let $D_1(a_1,\dots ,a_k)$ denote the set of all immediate descendants of $(a_1,\dots ,a_k)$. Note
that $\#D_1(a_1,\dots ,a_k)=k$.  A multiindex $(c_1,\dots ,c_k)\in \NN^k$ is an {\it ancestor} (resp. a {\it
parent}) of $(a_1,\dots ,a_k)$ if $c_i\le a_i$ for all $i$ and $c_i\ne a_i$ for at least one $i$ (resp. $c_i\le a_i$ for all
$i$ and $\sum _{i=1}^{k} (a_i-c_i)=1$). Let  $\mathcal A(c_1,\dots ,c_k)$  (resp. $\mathcal A_1(c_1,\dots ,c_k)$) denote the set of all ancestors
(resp. parents) of $(c_1,\dots ,c_k)$. Note that $\mathcal A(c_1,\dots ,c_k)=\emptyset$
if and only if $c_i\neq 0$ for all $i$ and
$\# \mathcal A(a_1,\dots ,a_k)$ is the number of $i\in \{1,\dots ,k\}$ such that $a_i\ne 0$.

\begin{proposition}\label{0bg1}
Fix a zero-dimensional scheme $Z\subset X:= \PP^{n_1}\times \cdots \times \PP^{n_k}$, $i\in \{1,\dots ,k\}$ and $(a_1,\dots ,a_k)\in 	\II_1(Z)$. Then:
\begin{enumerate}
\item $h^1(\Ii _Z\otimes (a_1,\dots ,a_k)+\epsilon _i))\le h^1(\Ii_Z(a_1,\dots ,a_k))$.
\item $h^1(\Ii _Z\otimes (a_1,\dots ,a_k)+\epsilon _i))=h^1(\Ii_Z(a_1,\dots ,a_k))$  if there is $p\in \PP^{n_i}$ such that
$Z\subset \pi _i^{-1}(p)$.
\item We have $\lim _{t\to +\infty} h^1(\Ii _Z\otimes (a_1,\dots ,a_k)+t\epsilon _i))>0$ if and only if there is $p\in
\PP^{n_i}$ such that $h^1(\Ii_{\pi _i^{-1}(p)\cap Z}(a_1,\dots ,a_k)) >0$.\end{enumerate}\end{proposition}

\begin{proof}
Fix a hyperplane $H$ of $\PP^{n_i}$ such that $H\cap \pi _i(Z)=\emptyset$. Therefore  $D:= \pi _1^{-1}(H)\in
|\Oo_X(\epsilon _i)|$  and
$Z\cap D =
\emptyset$. Thus we have an exact sequence
\begin{equation}\label{eq0b1}
0 \to \Ii _Z(a_1,\dots ,a_k)\to \Ii _Z((a_1,\dots ,a_k)+\epsilon _i)\to \Oo_D((a_1,\dots ,a_k)+\epsilon _i)\to 0
\end{equation}
The long cohomology exact sequence of \eqref{eq0b1} gives $h^1(\Ii _Z\otimes (a_1,\dots ,a_k)+\epsilon _i)\le
h^1(\Ii_Z(a_1,\dots ,a_k))$. Obviously,  equality holds  if there is $p\in \PP^{n_i}$ such that $Z\subset \pi _i^{-1}(p)$ (the
variables $x_{ij}$, $0\le j\le n_i$ are constant on $Z$).

Assume the existence of $p\in \PP^{n_i}$ such that $h^1(\Ii_{\pi _i^{-1}(p)\cap Z}(a_1,\dots ,a_k)) >0$. Set $W:= \pi
_1^{-1}(p)\cap Z$. Since $W\subset \pi_i^{-1}(p)$, then $h^1(\Ii_W(a_1,\dots ,a_k)=h^1(\Ii _W\otimes (a_1,\dots ,a_k)+t\epsilon
_i)$ for all $t>0$. Using Remark \ref{0bg01} we get the `` if '' part of (3). Now assume that there is no such $p$. Thus $Z\ne
\emptyset$ and there is  $p\in \PP^{n_i}$ such that $L:= \pi _i^{-1}(p)$ meets $Z$, but does not contain $Z$. Hence there is a
hyperplane $H'\subset \PP^{n_i}$ such that $p\in H'$, but $\pi _i(Z)\nsubseteq H'$. Set $H:= \pi _i^{-1}(H')\in
|\Oo_X(\epsilon _i)|$. To prove part (3) we use induction on the integer
$\deg (Z)$. Consider the residual exact sequence

\begin{equation}\label{eq0b2}
0 \to \Ii _{\Res_H(Z)}((a_1,\dots ,a_k)+(t-1)\epsilon _i)\to \Ii _Z((a_1,\dots ,a_k)+t\epsilon _i) \to \Ii _{Z\cap
H,H}((a_1,\dots ,a_k)+t\epsilon _i)\to 0
\end{equation}
By assumption $h^1(H,\Ii _{Z\cap H,H}((a_1,\dots ,a_n)+t\epsilon _i)) =0$ for all $t\in \NN$. Since $\Res_H(Z)\subset Z$, 
Remark \ref{0bg01} gives  $h^1(\Ii_{\pi _i^{-1}(p)\cap \Res_H(Z)}(a_1,\dots ,a_k)) =0$  for all $p\in \PP^{n_i}$. Since $\deg
(\Res_H(Z))=\deg (Z)-\deg (Z\cap H)<\deg (Z)$ the inductive assumption gives $h^1(\Ii _{\Res_H(Z)}((a_1,\dots
,a_k)+(t-1)\epsilon _i)) =0$ for all $t\gg 0$. The long cohomology exact sequence of \eqref{eq0b1} conclude the proof of part
(3).
\end{proof}

\begin{remark} Part (1) of Proposition \ref{0bg1} is a generalization of Proposition 3.5 in \cite{VT1} and Lemma 3.4 in \cite{msMultigraded}.
\end{remark}

Part (1) of Proposition \ref{0bg1} may be extended to this more general set-up.

\begin{lemma}\label{abg1}
Let $Y$ be an integral projective variety and $R$ a line bundle on $Y$ such that $h^1(R)=0$ and $R$ is globally generated. Set
$M:= Y
\times \PP^m$ and call $\pi_1: M\to Y$ and $\pi_2: M\to \PP^m$ the projections. For any $t\in \ZZ$ set $R\boxtimes (t): =
\pi_1^\ast(R)\otimes \pi_2^\ast(\Oo _{\PP^m}(t))$. Let $Z\subset M$ be a zero-dimensional scheme such that there is $z\in \NN$
with $h^1(\Ii _Z\otimes R\boxtimes (z))=0$ and call $e$ the minimal such a natural number. Then $h^1(\Ii_Z\otimes R\boxtimes
(y)) =0$ for all $y\ge e$.
\end{lemma}

\begin{proof}
By the K\"{u}nneth formula  $H^0(R\boxtimes (t)) \cong H^0(R)\otimes H^0(\Oo _{\PP^1}(t))$ and $H^1(R\boxtimes (t)) \cong H^0(R)\otimes
H^1(\Oo_{\PP^m}(t)) \oplus H^1(R)\otimes H^0(\Oo_{\PP^m}(t))$. Thus $h^1(R\boxtimes (t)) =0$ for all $t\ge -1$. 

By the definition of the natural number $e$ the lemma is true for the integer $y=e$. Thus to prove part (1) we may assume $y>e$
and that part (1) is true for all integers
$t\in \{e,\dots ,y-1\}$. Since $R$ is globally generated and $Z_{\red}$ is a finite set, there is $\sigma \in H^0(R)$ not
vanishing at any $p\in u _1(Z_{\red})$. Thus $\sigma':= \pi_1^\ast (\sigma)$ vanishes at no point of $Z_{\red}$. Since
$\pi_2(Z_{\red})$ is finite, there is a hyperplane $H\subset \PP^m$ such that $H\cap \pi_2(Z)=\emptyset$. See $H$ as an element of
$|\Oo_{\PP^m}(1)|$ and call
$\tau$ a section of $H^0(\Oo_{\PP^m}(1))$ vanishing on $H$. Set $u:= \sigma'\boxtimes \pi_2^\ast (\tau)$. Note that $u\in
H^0(R\boxtimes (1))$ and vanishes at no point of $Z_{\red}$. Since $h^1(R\boxtimes (y-1)) =0$, $Z$ imposes $\deg (Z)$
independent conditions to $H^0(R\boxtimes (y-1))$. Since $u$ vanishes at no point of $Z_{\red}$, the multiplication by $u$
shows that $Z$ imposes $\deg (Z)$ independent conditions to $H^0(R\boxtimes (y))$.\end{proof}

\begin{remark}\label{ad1} Obviously every descendant of an element of $\II_0(Z)$ is an element of $\II_0(Z)$. Part (1) of
Proposition \ref{bg1} says that any descendant of an element of $\NN^k\setminus \II_1(Z)$ is an element of $\NN^k\setminus
\II_1(Z)$. If $a_i\ge \deg (Z)-1$ for all $i$, then $(a_1,\dots ,a_k)\notin \II_1(Z)$. Part (2) of Proposition \ref{bg1} says that
if $(a_1,\dots ,a_k)-\epsilon _i\notin \II_1(Z)$, then the multiplication map $H^0(\Ii _Z(a_1,\dots ,a_k))\otimes
H^0(\Oo_X(\epsilon _i)) \to H^0(\Ii _Z(a_1,\dots ,a_k)+\epsilon_i)$ is surjective and hence we do not need to take any element of $H^0(\Ii
_Z(a_1,\dots ,a_k)+\epsilon_i)$ to generate the multigraded ideal $\II_0(Z)$.
\end{remark}

\section{Case $X := \mathbb P^{1}\times\cdots  \times\mathbb P^{1}$}\label{generators}

In this section we take $X := \mathbb P^{1}\times\cdots  \times\mathbb P^{1}:= (\PP^1)^k$. In particular, in Theorem \ref{bb1} we explicity describe the generators of the multigraded ideal defining a $0$-dimensional scheme.  We start with the following result.
\begin{proposition}\label{bg1}
Let $Y$ be an integral projective variety and $R$ a line bundle on $Y$ such that $h^1(R)=0$ and $R$ is globally generated. Set $X:= Y
\times \PP^1$ and call $\pi _1: X\to Y$ and $\pi _2: X\to \PP^1$ the projections. For any $t\in \ZZ$ set $R\boxtimes (t): = \pi_1^\ast(R)\otimes \pi_2^\ast(\Oo _{\PP^1}(t))$.
Let $Z\subset X$ be a zero-dimensional scheme such that there is $z\in \NN$ with $h^1(\Ii _Z\otimes R\boxtimes (z))=0$ and call $e$ the minimal such a natural number.
For any $t\ge 0$ let $\mu _t: H^0(\pi _2^\ast (\Oo _{\PP^1}(1)) \otimes H^0(\Ii _Z\otimes \Oo_Y\boxtimes (t)) \to H^0(\Ii
_Z\otimes R\boxtimes (t+1))$ denote the multiplication map. Then
\begin{enumerate}
\item $h^1(\Ii_Z\otimes R\boxtimes (y)) =0$ for all $y\ge e$.
\item $\mu_y$ is surjective for all $y>e$.
\item $\dim \mathrm{coker}(\mu _e) = h^1(\Ii _Z\otimes R\boxtimes (e-1))$.
\item If $e=0$, then $\dim \mathrm{coker}(\mu _e) =\deg (Z)$.
\end{enumerate}
\end{proposition}

\begin{proof}
By the K\"{u}nneth formula  $H^0(R\boxtimes (t)) \cong H^0(R)\otimes H^0(\Oo _{\PP^k}(t))$ and $H^1(R\boxtimes (t)) \cong
H^0(R)\otimes H^1(\Oo_{\PP^k}(t)) \oplus H^1(R)\otimes H^0(\Oo_{\PP^k}(t))$. Thus $h^1(R\boxtimes (t)) =0$ for all $t\ge -1$. 

By the definition of the natural number $e$ part (1) is true for the integer $y=e$. Thus to prove part (1) we may assume $y>e$
and that part (1) is true for all integers
$t\in \{e,\dots ,y-1\}$. Since $R$ is globally generated and $Z_{\red}$ is a finite set, there is $\sigma \in H^0(R)$ not
vanishing at any $p\in \pi _1(Z_{\red})$. Thus $\sigma':= \pi _1^\ast (\sigma)$ vanishes at no point of $Z_{\red}$. Since $\pi
_2(Z_{\red})$ is finite, there is $q\in \PP^1\setminus \pi _2(Z_{\red})$. See $q$ as an element of $|\Oo_{\PP^1}(1)|$ and call
$\tau$ a section of $H^0(\Oo_{\PP^1}(1))$ vanishing on $q$. Set $u:= \sigma'\otimes \pi _2^\ast (\tau)$.
Note that $u\in
H^0(R\boxtimes (1))$ and vanishes at no point of $Z_{\red}$. Since $h^1(R\boxtimes (y-1)) =0$, $Z$ imposes $\deg (Z)$
independent conditions to $H^0(R\boxtimes (y-1))$. Since $u$ vanishes at no point of $Z_{\red}$, the multiplication by $u$
shows that 
 $Z$ imposes $\deg (Z)$ independent conditions to $H^0(R\boxtimes (y))$, concluding the proof of part (1).
 
 Consider the following exact sequence of vector bundles on $\PP^1$:
 \begin{equation}\label{eqb1}
 0 \to \Oo _{\PP^1}(-1) \to \Oo_{\PP^1}^{\oplus 2}\to \Oo_{\PP^1}(1)\to 0
 \end{equation}
Taking $\pi _2^\ast$ of \eqref{eqb1} we get an exact sequence of locally free sheaves on $X$. Thus twisting this exact sequence by $\Ii_Z\otimes R\boxtimes (y-1)$ we get the exact sequence
\begin{equation}\label{eqb2}
 0 \to \Ii _Z\otimes  R\boxtimes (y-2) \to (\Ii _Z\otimes R\boxtimes (y-1))^{\oplus 2}\to \Ii _Z\otimes R\boxtimes (y)\to 0
 \end{equation}
The $H^0$ of the surjection in \eqref{eqb2} is the map $\mu _y$. Part (1) and \eqref{eqb2} gives part (2). Now take $y=e+1$. The definition of $e$ gives $h^1(\Ii _Z\otimes  R\boxtimes (e)) =0$. If $e>0$ the definition of $e$ gives  $h^1(\Ii _Z\otimes  R\boxtimes (e-1)) \ne 0$. Thus \eqref{eqb1} gives part (3) for $e>0$. 
Now assume $e=0$. By assumption $h^1(\Ii _Z\otimes R\boxtimes (0)) =0$. Thus $\pi _1: X \to Y$ induces an embedding $\pi _{1|Z}: Z \to  Y$ and $h^1(Y,\Ii _{\pi _1(Z)}\otimes R)=0$. Since $H^0(R\boxtimes (-1)) =0$, $h^1(\Ii _Z\otimes R\boxtimes (-1))=\deg (Z)$. Thus (4) is true and prove the case $e=0$ of (3).
\end{proof}

\begin{theorem}\label{bb1}
Let $Z\subset X =(\PP^1)^k$ be a zero-dimensional scheme with maximal rank. Set $z:= \deg (Z)$. For any $(a_1,
\dots a_k)\in \II_0(Z)$ and any $i\in \{1,\dots ,k\}$ take a general linear subspace $W(a_1,\dots,a_k;i)\subset H^0(\Ii _Z((a_1,\dots ,a_k)+\epsilon _i))$ such that
$\dim W(a_1,\dots,a_k;i) = z+(a_i+2)\Delta _i -2\Delta$, where $\Delta:= \prod _{i=1}^{k} (a_i+1)$ and for any $S\subseteq \{1,\dots ,k\}$ $\Delta _S:= \Delta/(\prod _{i\in S}(a_i+1))$. Then

\quad (i) The multiplication map $\mu: H^0(\Oo _X(\epsilon _i)\otimes H^0(\Ii_Z(a_1,\dots ,a_k))\to H^0(\Ii_Z((a_1,\dots ,a_k)+\epsilon _i))$ has cokernel of dimension $z+(a_i+2)\Delta _i -2\Delta$
and $H^0(\Ii_Z((a_1,\dots ,a_k)+\epsilon _i))$ is generated by this cokernel and $W(a_1,\dots,a_k;i)$.

\quad (ii) The multigraded ideal $\II(Z)$ of $Z$ is generated by the direct sum $\TT$ of all $H^0(\Ii_Z(a_1,\dots ,a_k))$ and all $W(a_1,\dots,a_k;i)$, $(a_1,\dots ,a_k)\in \II_0(Z)$, $i\in \{1,\dots ,k\}$, except that if $(a_1,\dots ,a_k)+\epsilon _i
= (b_1,\dots ,b_k)+\epsilon _j$ for some others $(b_1,\dots ,b_k)\in \II_0(Z)$ we only take one subspace $W$ among all possible $(b_1,\dots ,b_k)$ and $\epsilon _j$.

\quad (iii) With the restriction on $\TT$ given in  part (ii) no proper subspace of $\TT$ generates $\II(Z)$.\end{theorem}
\begin{proof}
In the statement of the theorem we wrote $\Delta _i$ instead of $\Delta _{\{i\}}$ and we will use this short-hand in the proof. Note that all $\Delta _S$ are positive integers. 

We first prove parts (i) and  (iii). 

{\em Step (a)} A basis of each $H^0(\Ii_Z(a_1,\dots ,a_k))$, $(a_1,\dots ,a_k)\in \II_0(Z)$, is contained in any minimal set of multigraded generators of $\II(Z)$ (Remark \ref{ad1}).

Since $Z$ has maximal rank and $(a_1,\dots a_k)\in \II_0(Z)$, $h^1(\Ii _Z(a_1,\dots ,a_k)) =0$, $h^0(\Ii _Z(a_1,\dots ,a_k)) =\Delta -z>0$, $h^0(\Ii _Z(a_1,\dots ,a_k)-\epsilon _i) =0$
and $h^1(\Ii _Z(a_1,\dots ,a_k)-\epsilon _i) =z-a_i\Delta/(a_i+1)$. We consider the exact sequence
\begin{equation}\label{eqbb1}
0 \to \Ii_Z\otimes ((a_1,\cdots ,a_k)-\epsilon _i)\to H^0(\Oo _X(\epsilon _i)\otimes \Ii _Z(a_1,\dots a_k)\to \Ii_Z\otimes ((a_1,\cdots ,a_k)+\epsilon _i)\to 0
\end{equation}
Since $h^0(\Ii_Z\otimes ((a_1,\cdots ,a_k)-\epsilon _i))=0$, $\mu$ is injective and hence
$\dim \mathrm{Coker}(\mu) = (a_i+2)\Delta _i -z -2(\Delta -z) = z+(a_i+2)\Delta _i -2\Delta$. Since $W(a_1,\dots,a_k;i)$ is  general and $\dim W(a_1,\dots ,a_k;i) =\dim \mathrm{Coker}(\mu)$, we have  $H^0(\Ii_Z\otimes ((a_1,\cdots ,a_k)+\epsilon _i))=
\mathrm{Im}(\mu) +W(a_1,\dots ,a_k;i)$ and $\mathrm{Im}(\mu) \cap W(a_1,\dots ,a_k)=\{0\}$. Thus no proper subspace of $W(a_1,\dots,a_k;i)$ together with $\mathrm{Im}(\mu)$ spans  $H^0(\Ii_Z\otimes ((a_1,\cdots ,a_k)+\epsilon _i))$.
The restriction we gave in $\TT$ on the $W$'s we take means that no proper subset of a multigraded basis of $\TT$ generates $\II(Z)$, concluding the proof of part (iii).

{\em Step (b)} We prove part (ii) and, hence, we conclude the proof of the theorem. 

Let $\UU$ be the linear subspace of $\II(Z)$ generated by $\TT$. By part (2) of Proposition \ref{bg1} it is sufficient to prove that $\UU$ contains
all linear spaces $H^0(\Ii _Z(b_1,\dots ,b_k))$ such that there is $(a_1,\dots ,a_k)\in \II_0(Z)$ with $a_i\le b_i\le a_i+1$ for all $i$ and $(b_1,\dots ,b_k)\ne (a_1,\dots ,a_k)$. Fix any such $(b_1,\dots ,b_k)$ and set $S:= \{i\in \{1,\dots ,k\}\mid b_i=a_i+1\}$.
Step (a) proved the case $\#S =1$. Thus we may assume $s:= \#S \ge 2$ and use induction on the integer $s$. We order the elements $\{i_1,\dots ,i_s\}$ of $S$ so that $a_{i_1}\le \cdots \le a_{i_s}$. Set $S':= S\setminus \{i_s\}$.
Set $c_i:= b_i$ for all $i\ne i_s$ and $c_{i_s} =a_{i_s}$. By the inductive assumption $H^0(\Ii _Z(c_1,\dots,c_k))\subset \UU$. Thus is it sufficient to prove that the multiplication map $\eta : H^0(\Oo_X(i_s))\otimes H^0(\Ii_Z(c_1,\dots ,c_k))
\to H^0(\Ii _Z(b_1,\dots ,b_k))$ is surjective. Since $a_{i_s} \ge a_{i_1}$ and $b_{i_1} =a_{i_1}+1$ we have $\prod _{h=1}^{s-1} (b_{i_h}+1)\times a_{i_s}\ge \prod _{h=1}^{s} (a_h+1)$. Thus $\Delta_S\times \prod _{h=1}^{s-1} (b_{i_h}+1)\times a_{i_s}\ge \Delta$.
Since $Z$ has maximal rank and $z< \Delta$, $h^1(\Ii_Z((c_1,\dots ,c_k))-\epsilon _{i_s}))=0$. Part (2) of Proposition \ref{bg1} gives that $\eta$ is surjective.
\end{proof}



\section{General subsets of some multiprojective spaces} \label{generators2}

In this section we take  $X:= \PP^{n_1}\times \cdots \times \PP^{n_k}$ with $n_i\in \{1,2\}$ for all $i$. In particular, in Theorem \ref{p2p1} we explicity describe the generators of the multigraded ideal defining a $0$-dimensional scheme.  We start with the following result.
\begin{remark}\label{kk2}
Let $Y$ be an integral projective variety, $M$ an effective Cartier divisor of $Y$, $E_1$ a vector bundle on $Y$ and $E_2$ a
vector bundle on $M$. Assume the existence of a surjection $\phi: E_{1|M} \to E_2$ and call $\psi$ the composition of $\phi$ with
the restriction map $E_1\to E_{1|M}$. Note that $\psi : E_1\to E_2$ is a surjection of coherent $\Oo_Y$-sheaves in which we see the $\Oo_M$-sheaf $E_2$ as a
torsion $\Oo_Y$-sheaf. Set
$E_0:=
\mathrm{ker}(\psi)$. $E_o$ is called the {\it elementary transformation} of $E_1$ with respect to $\phi$ (\cite[beginning of \S 2]{hs}, \cite[Th. 1.4]{mar}). It is known that
$E_0$ is locally free (\cite{mar}). Set $r:= \mathrm{rank}(E_1)$ and $k:= \mathrm{rank}(E_2)$. Since $\phi$ is surjective, $r\ge k$ and $E_3:= \mathrm{ker}(\phi)$ is a rank $(r-k)$ vector bundle on $M$, with $E_3$ the zero-sheaf if $r=k$, which occurs if and only if $E_2=E_{1|M}$. The rank $r$ vector bundle $E_{0|M}$ on $M$ fits in the exact sequence
\begin{equation}\label{eqke2}
0 \to E_2(-M) \to E_{0|M} \to E_3\to 0
\end{equation}
 (\cite[beginning of \S 2]{hs}, \cite[Theorem 1.4]{mar}) in which $E_2(-M):= E_2\otimes _{\Oo_D}\Oo_D(-D)$.
Fix an integer $s>0$. Suppose you want to prove that $h^0(\Ii _S\otimes E) =\max \{0,h^0(E_1)-rs\}$ for a general $S\subset Y$ with $\#S=s$. Suppose you want to prove it using the semicontinuity theorem for cohomology specializing a general $S$ to a general $A\cup B$ with $A$ general in $Y$,  $B$ general in $M$ and $\#B$ as high as possible. The problem is that in the set-up we need $E_{1|M} \cong E_2\oplus F$ with $F$ a rank  $(r-k)$ vector bundle and $h^0(E_2)/k < h^0(F)/(r-k)$. In our set-up we will have $k=1$ and hence  $E_2$ is a line bundle on $M$ with $h^1(M,E_2)=0$. Thus to get $h^i(D,\Ii _{B,D}\otimes F)=0$, $i=0,1$, it is sufficient to take  $B$ as a general subset of $M$ with cardinality $h^0(D,M)$. However, $h^0(M,\Ii_{B,M}\otimes F')
\ge h^0(F)/(r-1)>0$
 and hence to conclude it is not sufficient to use that $h^0(\Ii _A\otimes E_1(-M)) =
 \max \{0,h^0(E_1(-M)) -r(s-h^0(E_2))$ for a general $A\subset Y$ with $\#A=s-h^0(E_2)$. 
Set $\VV:= \PP E_0$ and call $\Oo_{\VV}(1)$ the tautological line bundle on $\VV$  and $\pi : \VV\to Y$ the projection. We have $h^0(\Oo_{\VV}(1))) =h^0(E_0)$. For any $p\in Y$, $\pi^{-1}(p)$ is an $(r-1)$-dimensional projective space and $h^0(Y,\Ii _p\otimes E_0) = h^0(\VV,\Ii_{\pi^{-1}(p)}(1))$ and a similar relation holds for any finite subset of $Y$. Thus a point of $A$ should give $r$ independent conditions to $H^0(\VV,\Oo _{\VV}(1))$.
 $B$ gave $h^0(E_2)$ independent conditions to $H^0(E_2)$ and it should give  $(r-b)$ further conditions to $H^0(\VV,\Oo _{\VV}(1))$. In \cite[Lemma 5.1.1]{ida}, it is $k=1$ and $r=3$. From now on we assume $k=1$. We will only need the case $r=2$. Assume $r=2$. As in \cite{ida}, we call the point of $\VV$ corresponding to some $p\in M$ an s-point. To prove that $h^0(\Ii _S\otimes E_1) =\max \{0,h^0(E)-2s\}$ it would be sufficient to prove that $h^0(\VV,\Ii _{A\cup B'}(1)) = \max \{0,h^0(E_0)-2\#A-h^0(E_2)$, where $B'$ is a general union of $h^0(E_2)$  s-points.
\end{remark}

\begin{remark}\label{ke3}
The Euler's sequence 
$$0 \to \Omega ^1_{\PP^2}(1) \to \Oo_{\PP^2}^{\oplus 3} \to \Oo_{\PP^2}(1)\to 0$$ gives $h^1(\Omega ^1_{\PP^2}(x)) =0$ for all $x\ne 0$, $h^0(\Omega ^1_{\PP^2}(x)) =0$ for all $x\le 1$ and $h^0(\Omega ^1_{\PP^2}(x)) =3\binom{x+1}{2} -\binom{x+2}{2}
= x^2-1$ for all $x\ge 2$. Let $C\subset \PP^2$ be a smooth conic. By \cite[Lemma 1.3]{b0} $\Omega ^1_{\PP^2}(x)_{|C}$ is a direct sum of two line bundles of degree $2x-3$.
Thus $h^0(C,\Ii _S\otimes \Omega ^1_{\PP^2}(x)_{|C}) =\max\{0,2x-1-2s\}$ for all $x\ge 2$ and $h^1(C,\Ii _S\otimes \Omega ^1_{\PP^2}(x)_{|C}) =\max\{0,2s-2x+2\}$ for
any finite set $S\subset C$ such that $\#S =s$.
\end{remark}

\begin{proposition}\label{bg2}
Let $Y$ be an integral projective variety and $R$ a line bundle on $Y$ such that $h^1(R)=0$, $\alpha:= h^0(R)>0$ and $R$ is globally generated. Set $X:= Y
\times \PP^2$ and call $\pi _1: X\to Y$ and $\pi _2: X\to \PP^2$ the projections. For any $t\in \ZZ$ set $R\boxtimes \Omega ^1_{\PP^2}(t): = \pi_1^\ast(R)\otimes \pi_2^\ast(\Omega ^1_{\PP^2}(t))$.

\quad (a) We have $h^0(R\boxtimes \Omega ^1_{\PP^2}(x))=\alpha(x^2-1)$ for all integers $x\ge 2$, $h^0(R\boxtimes \Omega ^1_{\PP^2}(1)) =0$ and  $h^1(R\boxtimes \Omega ^1_{\PP^2}(x))=0$ for all $x\ge 1$.

\quad (b) Fix an integer $s>0$ and let $S\subset X$ be a general subset with cardinality $s$. Then $h^0(\Ii_S\otimes R\boxtimes \Omega ^1_{\PP^2}(x))=\max \{0,\alpha(x^2-1)-2s\}$ and  $h^1(\Ii_S\otimes R\boxtimes \Omega ^1_{\PP^2}(x))=\max \{0,2s-\alpha(x^2-1)\}$ for all $x\ge 1$.
\end{proposition}

\begin{proof}
Part (a) follows from the first part of Remark \ref{ke3}
and the K\"{u}nneth formula. Since $R\boxtimes \Omega ^1_{\PP^2}(x)$ has rank $2$, for any finite union $S$ of $s$ points we have $h^0(\Ii_S\otimes R\boxtimes \Omega ^1_{\PP^2}(x)) 
\ge \max \{0, \alpha(x^2-1) -2s\}$, $h^1(\Ii_S\otimes R\boxtimes \Omega ^1_{\PP^2}(x)) \ge\max \{0,
2s-\alpha(x^2-1)\}$ and one of these inequalities is an equality if and only if the other one is an equality.

Set $s'(x): = \lfloor
(x^2-1)\alpha/2\rfloor$ and
$s''(x):= \lceil (x^2-1)\alpha/2\rceil$. 

\quad {\bf Claim 1:} The lemma is true for a fixed $x$ and all $s$ if and only if it is true for the pairs $(x,s'(x))$
and $(x,s''(x))$.

\quad {\bf Proof Claim 1:} The `` only if '' part is trivial. Assume that the lemma is true for the pairs $(x,s'(x))$
and $(x,s''(x))$. Thus $h^1(\Ii _A\otimes R\boxtimes \Omega ^1_{\PP^2}(x))=0$ for a general $A\subset X$ such that
$\#A = s'(x)$ and $h^0(\Ii _B\otimes R\boxtimes \Omega ^1_{\PP^2}(x)) =0$ for a general $B\subset X$ such that
$\#B = s''(x)$. Fix
$s\in
\NN$. First assume
$s\le s'(x)$. Take $E\subseteq A$ such that $\#E=s$. Since $E\subseteq A$ and $h^1(\Ii _A\otimes R\boxtimes \Omega ^1_{\PP^2}(x))=0$, $h^1(\Ii _E\otimes R\boxtimes \Omega ^1_{\PP^2}(x))=0$. Thus the lemma is true for the triple
$(x,s)$ by the semicontinuity theorem for cohomology. Now assume $s\ge s''(x)$. Take any finite set $F\supseteq B$ such
that $\#B=s$. Since $h^0(\Ii _B\otimes R\boxtimes \Omega ^1_{\PP^2}(x)) =0$,  it is $h^0(\Ii _F\otimes R\boxtimes \Omega ^1_{\PP^2}(x)) =0$. Apply the semicontinuity theorem for cohomology.

We divide the remaining parts of the proof into $3$ steps. 

\quad {\em Step (i)} Assume $x$ odd. Thus $x^2-1$ is even. Hence $s'(x)=s''(x) = (x^2-1)\alpha/2$. The case $x=1$ is true for all $s$, because $h^i(\Ii_S\otimes R\boxtimes \Omega ^1_{\PP^2}(1))=0$, $i=0,1$. Assume $x\ge 3$ and odd and that the proposition is true for the integer $x-2$. Let $C\subset \PP^2$ be a smooth conic and set $D:= \pi _2^{-1}(C)$. Fix a general $A\subset  X\setminus D$ with $\#S = s'(x-2)$ and a general $E\subset D$ such that $\#E = s'(x)-s'(x-2)$. Set $B=E\cup D$. To prove the proposition for the integer $s$ it is sufficient to prove that $h^0(\Ii_B\otimes R\boxtimes \Omega ^1_{\PP^2}(x)) =0$.  Consider the residual exact sequence
\begin{equation}\label{eqbg1a}
0 \to \Ii _S\otimes R\boxtimes \Omega ^1_{\PP^2}(x-2) \to \Ii _B\otimes R\boxtimes \Omega ^1_{\PP^2}(x) \to \Ii _{E,D}\otimes R\boxtimes \pi _2^{-1}(\Oo_{\PP^2}(x)|C) \to 0
\end{equation}
The inductive assumption gives $h^0(\Ii _S\otimes R\boxtimes \Omega ^1_{\PP^2}(x-2))=0$. 

\quad {\bf Claim 2:} $h^i(D,\Ii_{E,D}\otimes R\boxtimes \pi_2^\ast (\Omega ^1_{\PP^2}(x)_{|C}))=0$, $i=0,1$.

\quad {\bf Proof Claim 2:} Remark \ref{ke3} and the K\"{u}nneth formula give $2\#E = h^0(R\boxtimes \pi_2^\ast (\Omega ^1_{\PP^2}(x)_{|C})$. Since $h^0(\Oo_D(x,y)\otimes \pi _1^\ast (\Omega ^1_{\PP^2}))=2(\#E)$ and $\Omega ^1_{\PP^2}$ has rank $2$,
$h^0(D,\Ii_{E,D}\otimes R\boxtimes \pi_2^\ast (\Omega ^1_{\PP^2}(x)_{|C}))= h^1(D,\Ii_{E,D}\otimes R\boxtimes \pi_2^\ast (\Omega ^1_{\PP^2}(x)_{|C}))$. By Remark \ref{ke3} $\pi _2^\ast (\Omega ^1_{\PP^2})_{|D}$ is isomorphic
to the direct sum of two isomorphic line bundles on $X$. Thus  Claim 2 follows from the observation that for any integral projective variety $Y$ and any line bundle $M$ on $Y$
$h^0(Y,\Ii_G\otimes M)=0$ for a general $G\subset Y$ such that $\#G = h^0(M)$.
By Remark \ref{ke3}, the K\"{u}nneth formula and the generality of $E$, $h^i(D,\Ii _{E,D}\otimes R\boxtimes \pi _2^{-1}(\Oo_{\PP^2}(x)|C))=0$, $i=0,1$. 

To prove the proposition for $x$ use the cohomology exact sequence of \eqref{eqbg1a} and Claims 1 and 2.

\quad { \em Step (ii)} Assume $x$ even and $x\ge 4$. We will not prove that the lemma is true in this case, we will only prove that it is true
if it is true the case $x'=2$, which will be proved in Step (iii). The inductive proof $x-2\Rightarrow x$ is done as in Step (i) using a  smooth conic $C\subset \PP^2$ and \eqref{eqbg1a} first for $s'(x-2)$ and $s'(x)$ and then for $s''(x-2)$ and $s''(x)$.

\quad {\em Step (iii)} Assume $x=2$. By Steps (i) and (ii) it is sufficient to prove this case to conclude the proof of the proposition. We recall that $s'(2)=\lfloor 3\alpha/2\rfloor$ and $s''(2) =\lceil 3\alpha/2\rceil$ and
the definition of elementary transformation $E_0$ of a vector bundle $E_1$ with respect to a surjection $\phi :E_{1|M} \to E_2$ with  $M\subset X$ an effective divisor and $E_2$ a vector bundle on $M$. In this step $E_1$ has rank $2$ and $E_2$ is a line bundle on $M$
(Remark \ref{kk2}). Recall that $h^0(R\boxtimes \Omega ^1_{\PP^2}(2))=3\alpha$. Fix a line $L\subset \PP^2$ and set
$M:= \pi_1^{-1}(L)$. Since $\Omega ^1_{\PP^2}(2)_{|L} \cong \Oo_L(1)\oplus \Oo_L$, there is a surjection $\rho':
\Omega ^1_{\PP^2}(2)\to \Oo_L$ with $\mathrm{ker}(\rho') \cong \Oo_{\PP^2}^{\oplus 2}$ (\cite[Lemma 5.1.1]{ida}). Applying $\pi_2^\ast$ and twisting by  $\pi^*_1(R)$ we get a
surjection
$\psi: R\boxtimes \Omega ^1_{\PP^2}(2)
\to R\boxtimes \Oo_M$. 

\quad  {\bf Claim 3:} $\mathrm{ker}(\psi) \cong (R\boxtimes \Oo_{\PP^2})^{\oplus 2}$.

\quad {\bf Proof Claim 3:} Set $G:= \mathrm{ker}(\psi)$. Recall that \cite[Lemma 5.1.1]{ida} gives $\mathrm{ker}(\rho ') \cong \Oo_{\PP^2})^{\oplus 2}$, since $\psi$ is obtained twisting by $R$ the pull-back $\pi _1^\ast$ of $\rho'$. Since elementary transformations commute with pull-backs (\cite[Proposition 2.3]{mar}), we get Claim 3.

\quad {\em Step (iii-1)} We first prove the case $x=2$ of the lemma for the integer $s:= s'(2)$. We take a set $S\subset X$ with $\#S = s'(2)$ and with $S=A\cup B$, where $A$ is a general subset of $X$ with $\#A =s'(2)-h^0(R)$ and $B$ is a general subset of $M$ with $\#B = h^0(R)$. By semicontinuity to prove the lemma in this case it is sufficient to prove that $h^1(\Ii _S\otimes R\boxtimes \Omega ^1_{\PP^2}(2)) =0$. Since $B\subset M$ is general, $h^i(M,\Ii _{B,M}\otimes R\boxtimes \Oo_L) =0$, $i=0,1$. Claim 3 gives $h^1(\Ii _A\otimes G)=0$ and hence $h^0(\Ii_A\otimes G) =2\alpha-2(\#A)$. Take $\VV = \PP E_0$ As in \cite{hs,ida} to conclude the proof it is sufficient to prove that $yh^0(R)$ s-points of $\VV$ corresponding to $B$ gives $h^0(R)$ independent conditions to the vector space $H^0(\Ii_A\otimes G)$.
We order the points $p_1,\dots ,p_{h^0(R)}$ of $B$ and get an ordering $v_1,\dots ,v_{h^0(R)}$ of the $h^0(R)$ s-points of $\VV$ associated to $B$. Suppose the existence of $e\in \{1,\dots ,h^0(R)\}$ such that $v_e$ does not given an independent condition to the subspace $W$ of $H^0(\Oo_{\VV}(1))$ of all sections vanishing on $\pi^{-1}(A)\cup \{v_1,\dots ,v_{e-1}\}$ with the convention $\{v_1,\dots ,v_{e-1}\}=\emptyset$ if $e=1$. Note that $\dim W \ge h^0(R)+1-e>0$. Since $p_e$ is general in $M$ and $\pi _2^\ast (\Omega ^1_{\PP^2}(2))_{|M}$, we get that $v_e$ is a general $s$-point on $\VV_{|\pi ^{-1}(M)}$ residual of $\pi ^{-1}(p_e)$ by the elementary transformation $\psi$. Since $M =\pi_2^{-1}(L)$, $W$ injects in $H^0(R\boxtimes \pi _2^\ast (\Omega ^1_{\PP^2}(1)) =0$, a contradiction.

\quad {\em Step (iii-2)} To conclude the proof of the lemma it is sufficient to prove the case $(x,s) =(2,s''(2))$. It is sufficient to mimic the proof of Step (iii-1) taking as $A$ a general subset of $X$ with cardinality $s''(2)-h^0(R)$.
\end{proof}

\begin{lemma}\label{pbg1}
Take $Y$, $R$, $s$, $S$ and $t$ as in Proposition \ref{bg2} with $s<h^0(R)\binom{t+2}{2}$. Then the multiplication map 
$$\mu : H^0(\pi_2^\ast(\Oo _{\PP^2}(1)) \otimes H^0(\Ii _S\otimes R\boxtimes (t+1)) \to H^0(\Ii _S\otimes R\boxtimes (t+2))$$
is surjective.
\end{lemma}

\begin{proof}
If $Y$ is a point, the Euler's sequence of $\Omega ^1_{\PP^2}$ (Remark \ref{ke3}) shows that it is sufficient to prove that $h^1(\Ii _S\otimes \Omega ^1_{\PP^2}(t+2)) =0$. This is true by the generality of $S$ and \cite{gm} (i.e. the case $n=2$ of \cite{tv}).
The case $\dim Y>0$ is done (using the case just done) as in the proof of Proposition \ref{bg2}.
\end{proof}

\begin{theorem}\label{p2p1}
Take $X:= \PP^{n_1}\times \cdots \times \PP^{n_k}$ with $n_i\in \{1,2\}$ for all $i$. Fix an integer $z>0$ and take a general $S\subset X$ such that $\#S =z$. For each $(a_1,\dots ,a_k)\in \II_0(S)$ and each $i\in \{1,\dots ,n_k\}$
let $W(a_1,\dots ,a_k;i;n_i)$ be a general linear subspace of $H^0(\Ii_S((a_1,\dots ,a_k)+\epsilon _i))$ such that $\dim W(a_1,\dots ,a_k;i;n_i)=\max \{0,-z+ \Delta\binom{n_i+a_i+1}{n_i}/\binom{n_i+a_i}{n_i}-(n_i+1)(\Delta -z)$, where $\Delta:= \prod_{h=1}^{k} \binom{n_h+a_ih}{a_h}$.

\quad (i) $S$ has maximal rank. 

\quad (ii) The multigraded ideal $\II(S)$ of $S$ is generated by the direct sum $\TT$ of all $H^0(\Ii_S(a_1,\dots ,a_k))$ and all $W(a_1,\dots,a_k;i;n_i)$, $(a_1,\dots ,a_k)\in \II_0(S)$, $i\in \{1,\dots ,k\}$, except that if $(a_1,\dots ,a_k)+\epsilon _i
= (b_1,\dots ,b_k)+\epsilon _j$ for some others $(b_1,\dots ,b_k)\in \II_0(S)$ we only take one subspace $W$ among all possible $(b_1,\dots ,b_k)$ and $\epsilon _j$.

\quad (iii) With the restriction on $\TT$ given in  part (ii) no proper subspace of $\TT$ generates $\II(S)$.\end{theorem}

\begin{proof}
Since $h^1(\Ii _S(b_1,\dots ,b_k)) =0$ for each zero-dimensional scheme $S\subset X$ with $\deg (S)=z$ if $b_i\ge z-1$ for all $i$ and $S$ is general, then $S$ has maximal rank. Thus $\II_0(S)\cap \II_1(S) =\emptyset$.
Fix $(a_1,\dots ,a_k)\in \II_0(S)$ and $i\in \{1,\dots ,k\}$. Since $S$ has maximal rank and $(a_1,\dots ,a_k)\in A_0(S)$,
$h^1(\Ii _S(a_1,\dots ,a_k)) =0$, $h^1(\Ii _S((a_1,\dots ,a_k)+\epsilon _i) =0$, $h^0(\Ii _S(a_1,\dots ,a_k)) =\Delta -z$ and $h^0(\Ii _S((a_1,\dots ,a_k)+\epsilon _i) =-z+ \Delta\binom{n_i+a_i+1}{n_i}/\binom{n_i+a_i}{n_i} $. Call $\nu : H^0(\pi _i^\ast (\Oo _{\PP^{n_i}}(1))\otimes H^0(\Ii_S(a_1,\dots ,a_k))\to  H^0(\Ii_S((a_1,\dots ,a_k)+\epsilon _i))$ the multiplication map.
By Proposition \ref{bg1} (case $n_i=1$) and Proposition \ref{bg2} (case $n_i=2$) $\nu$ has maximal rank, i.e. $\dim \mathrm{coker}(\nu) = \max \{0,-z+ \Delta\binom{n_i+a_i+1}{n_i}/\binom{n_i+a_i}{n_i}-(n_i+1)(\Delta -z)$. Since $\dim W(a_1,\dots ,a_k;i;n_i)=\dim \mathrm{coker}(\nu)$ and $W(a_1,\dots ,a_k;i;n_i)$ is general, we have $\mathrm{Im}(\nu)+W(a_1,\dots ,a_k;i;n_i) = H^0(\Ii_S((a_1,\dots ,a_k)+\epsilon _i))$ and $\mathrm{Im}(\nu)\cap W(a_1,\dots ,a_k;i;n_i) =0$. By part (2) of Proposition \ref{bg1} and Lemma \ref{pbg1} to conclude the proof it is sufficient to test all components of $\II(S)$
with multidegree $(b_1,\dots ,b_k)$ such that there is $(a_1,\dots ,a_k)\in \II_0(S)$ with $a_i\le b_i\le a_i+1$ for all $i$. The case $b_h=a_h$ for $k-1$ indices $h$ was done using $\nu$. The step from this case to the case in which
$b_i=a_i+1$ for $i\in \{i_1,\dots ,i_s\}$ with $s\ge 2$ is done as in the last part of the proof of Theorem \ref{bb1}.
\end{proof}

\begin{lemma}\label{ee1}
Fix an integer $n\ge 2$. Then there exists an integer $t(n)$ such that for all integers $t\ge t(n)$ and all $s\in \NN$ either
$h^1(\PP^n,\Ii _S\otimes \Omega ^1_{\PP^n}(t)) =0$ or $h^0(\PP^n,\Ii _S\otimes \Omega ^1_{\PP^n}(t)) =0$, where $S$ is a
general subset of $\PP^n$ with $\#S=s$.
\end{lemma}

\begin{proof}
Recall that a particular case of \cite{hs} gives the existence of an integer $a(n)$ such that for all integers $t\ge 0$ and
$x\ge a(n)$ either $h^1(\PP^n,\Ii _A\otimes \Omega ^1_{\PP^n}(t)) =0$ or $h^0(\PP^n,\Ii _A\otimes \Omega ^1_{\PP^n}(t)) =0$,
where
$A$ is a general subset of $\PP^n$ with $\#A=x$. Let $t(n)$ be the minimal integer $\ge n+2$ such that $\binom{t(n)+n-2}{n} \ge
a(n)-1$. Fix integers $t\ge t(n)$ and $s\ge 0$ and let $S\subset \PP^n$ be a general subset with $\#S =s$. If $s\ge a(n)$,
then  either
$h^1(\PP^n,\Ii _S\otimes \Omega ^1_{\PP^n}(t)) =0$ or $h^0(\PP^n,\Ii _S\otimes \Omega ^1_{\PP^n}(t)) =0$ by the quoted result
of \cite{hs}. Now assume $s<a(n)$. Thus $s\le \binom{n+t(n)-21}{n}$. Since $S$ is general, $h^1(\PP^n,\Ii _S(x)) =0$ for all
$x\ge t-2$. Since $\dim S=0$, $h^i(S,\Oo _S(y)) =0$ for all $y\in \ZZ$ and all $i>0$. Thus the exact sequence$$0\to \Ii
_S(x)\to
\Oo_\PP^n(x)\to
\Oo_S(x)\to 0$$gives $h^i(\PP^n,\Ii _S(y)) =h^i(\PP^n,\Oo_{\PP^n}(y))$ for all $i\ge 2$ and all $y\in \ZZ$. Since
$h^1(\PP^n,\Ii _S(x)) =0$ for all
$x\ge t-2$, The Euler's sequence gives $h^1(\PP^n,\Ii _S\otimes \Omega ^1_{\PP^n}(t)) =0$.
\end{proof}

For any zero-dimensional scheme $Z\subset \PP^{n_1}\times \cdots \times \PP^{n_k}$, any $J\subseteq\{1,\dots ,k\}$ and any $(a_1,\dots ,a_k)\in \NN^k$ call for instance $\mu_{Z,J,a_1,\dots ,a_k}: H^0(\otimes _{i\in J}\pi _i^\ast (\Oo _{\PP^n_i}(1)))\otimes H^0(\Ii _Z(a_1,\dots ,a_k)) \to H^0(\Ii_Z(b_1,\dots ,b_k))$ the multiplication map, where $b_i=a_i+1$ if $i\in J$ and $b_i=a_i$ if $i\notin J$.

\begin{proposition}\label{ee2}
Fix $k$, $n_1,\dots ,n_k$, $i\in \{1,\dots ,k\}$, $s\in \NN$ and $(a_1,\dots ,a_k)\in \NN^k$ such that $a_i\ge t(n_i)-1$. Let $S\subset X:=\PP^{n_1}\times \cdots \times \PP^{n_k}$ be a general subset with $\#S =s$. 
Set $\tau _1:= \lfloor h^0(\Omega _{\PP^{n_i}}^1(a_i+1))/n_i\rfloor\times \prod _{j\ne i} \binom{n_j+a_j}{n_j}$
and $\tau_2:= \lceil h^0(\Omega _{\PP^{n_i}}^1(a_i+1))/n_i\rceil\times \prod _{j\ne i} \binom{n_j+a_j}{n_j}$. Assume
either $s\le \tau_1$ or $s\ge \tau_2$. Then
$\mu _{\{i\},S,a_1,\dots ,a_k}$ has maximal rank.
\end{proposition}

\begin{proof}
By the definition of $\mu _{\{i\},S,a_1,\dots ,a_k}$ and  the Euler's sequence it is sufficient to prove that either  $h^0(\Ii_S\otimes \pi _i^\ast (\Omega _{\PP^{n_i}}(a_i+1))\otimes \boxtimes _{j\ne i} \Oo_X(a_j))=0$ or $h^0(\Ii_S\otimes \pi _i^\ast (\Omega _{\PP^n_i}(a_i+1))\otimes \boxtimes _{j\ne i} \Oo_X(a_j))=0$. By the semicontinuity theorem for cohomology it is sufficient to find one subset $S$ such that $\#S=S$ and either  $h^0(\Ii_S\otimes \pi _i^\ast (\Omega _{\PP^n_i}(a_i+1))\otimes \boxtimes _{j\ne i} \Oo_X(a_j))=0$ or $h^0(\Ii_S\otimes \pi _i^\ast (\Omega _{\PP^{n_i}}(a_i+1))\otimes \boxtimes _{j\ne i}
\Oo_X(a_j))=0$.

\quad (a) Assume for the moment to have proved the theorem for the integers $s=\tau_1$ and $s=\tau_2$ and take $S_1$ (resp.
$S_2$)  such that
$\#S_1=\tau_1$ (resp. $\#S_2=\tau_2$) and either  $h^0(\Ii_{S_1}\otimes \pi _i^\ast (\Omega _{\PP^{n_i}}(a_i+1))\otimes \boxtimes
_{j\ne i}
\Oo_X(a_j))=0$ or
$h^0(\Ii_{S_1}\otimes \pi _i^\ast (\Omega _{\PP^{n_i}}(a_i+1))\otimes \boxtimes _{j\ne i} \Oo_X(a_j))=0$ (resp. either 
$h^0(\Ii_{S_2}\otimes \pi _i^\ast (\Omega _{\PP^{n_i}}(a_i+1))\otimes \boxtimes _{j\ne i}
\Oo_X(a_j))=0$ or
$h^0(\Ii_{S_2}\otimes \pi _i^\ast (\Omega _{\PP^{n_i}}(a_i+1))\otimes \boxtimes _{j\ne i} \Oo_X(a_j))=0$). By the
definition of
$\tau _1$ and $\tau _2$) we have $h^1(\Ii_{S_1}\otimes \pi _i^\ast (\Omega _{\PP^{n_i}}(a_i+1))\otimes \boxtimes _{j\ne i}
\Oo_X(a_j))=0$ and $h^0(\Ii_{S_2}\otimes \pi _i^\ast (\Omega _{\PP^{n_i}}(a_i+1))\otimes \boxtimes _{j\ne i} \Oo_X(a_j))=0$.
Assume for the moment $s<\tau_1$ and take any $S'\subset S_1$ such that $\#S'=s$. Since $h^1(\Ii_{S_1}\otimes \pi _i^\ast
(\Omega _{\PP^n_i}(a_i+1))\otimes \boxtimes _{j\ne i} \Oo_X(a_j))=0$,  $h^1(\Ii_{S'}\otimes \pi _i^\ast (\Omega
_{\PP^n_i}(a_i+1))\otimes \boxtimes _{j\ne i} \Oo_X(a_j))=0$. Now assume $s>\tau_2$. Take any $S''\supset S_2$ such that
$\#S''=s$. Since $h^0(\Ii_{S_2}\otimes \pi _i^\ast (\Omega _{\PP^{n_i}}(a_i+1))\otimes \boxtimes _{j\ne i} \Oo_X(a_j))=0$, 
$h^0(\Ii_{S'}\otimes \pi _i^\ast (\Omega _{\PP^{n_i}}(a_i+1))\otimes \boxtimes _{j\ne i} \Oo_X(a_j))=0$.

\quad (b) Assume $s=\tau_1$.  Permuting the factors of $X$ we may write $X \cong
\PP^{n_i}\times Y'$ with
$Y' =\prod _{j\ne i} \PP^{n_j}$. Take $S = A_1\times A_2$ with $A_1$ a general subset of $\PP^{n_i}$ with $\#S_1 =
\lfloor h^0(\Omega _{\PP^n_i}^1(a_i+1))/n_1\rfloor$ and $\#A_2=  \prod _{j\ne i} \binom{n_j+a_j}{n_j}$. By the generality of
$A_2$ we have
$h^1(Y',\Ii _{A_2}(b_1,\dots ,b_{k-1})) =0$, where $(b_1,\dots ,b_{k-1})$ is obtained from $(a_1,\dots ,a_k)$ deleting the
$i$-th entry
$a_i$. Since $A_1$ is general in $\PP^{n_i}$ and $a_i\ge t(n_i)-1$, Lemma \ref{ee1} gives $h^1(\PP^n{n_i},\Ii _{A_1}\otimes
\Omega ^1_{\PP^{n_i}}(a_i+1)) =0$. Since $S =A_1\times A_2$ we get $h^1(\Ii_S\otimes \pi _i^\ast (\Omega
_{\PP^{n_i}}(a_i+1))\otimes \boxtimes _{j\ne i} \Oo_X(a_j))=0$.

\quad (c) Assume $s=\tau_2$. We take $S +B_1\times A_2$ with $B_1$ a general subset of $\PP^{n_i}$ with  $\lceil h^0(\Omega
_{\PP^{n_i}}^1(a_i+1))/n_i\rceil$. Lemma \ref{ee1} gives $h^0(\PP^{n_i},\Ii _{B_1}\otimes \Omega _{\PP^{n_i}}^1(a_i+1)) =0$.
Repeat the proof of step (b).
\end{proof}

\section{Base point freeness}\label{BPfreeness}
Fix $X = \PP^{n_1}\times \cdots \times \PP^{n_k}$ with $k\ge 2$ and all $n_i>0$ for all $i$.

\begin{lemma}\label{f1}
Fix $(a_1,\dots ,a_k)\in \NN^k$ and let $Z\subset X$ be a zero-dimensional scheme such that $h^1(\Ii_Z(a_1,\dots ,a_k)) =0$. Set
$E:= Z_{\red}$. Fix $(c_1,\dots ,c_k)\in \NN^k$.

\quad (a) Fix $i\in \{1,\dots ,k\}$ and assume $c_j\ge a_j$ for all $j\neq i$ and $c_i>a_i$. Then
$|\Ii _Z(c_1,\dots ,c_k)|$ has no base points outside $\pi _i^{-1}(\pi _i(E)$.

\quad (b) Assume $c >a_i$ for all $i$. Then no point of $X\setminus E$ is a base point of $|\Ii _Z(c_1,\dots ,c_k)|$.

\end{lemma}

\begin{proof}
Since $\Oo_X(e_1,\dots ,e_k)$ has no base points for any $(e_1,\dots ,e_k)\in \NN^k$ to prove part (a) it is sufficient to
prove the case $c_j=a_j$ for all $j\ne i$ and $c_i=a_i+1$. Fix $o\in X\setminus \pi _i^{-1}(\pi _i(E)$. 
Fix a hyperplane $M$ of $\PP^{n_i}$ such that $\pi_i(o)\notin M$ and set $H:= \pi _i^{-1}(M)$. Since $H\cap Z=\emptyset$, we
have a residual exact sequence
\begin{equation}\label{eqf1}
0 \to \Ii_Z(a_1,\dots ,a_k)\to \Ii _{Z\cup \{o\}}((a_1,\dots ,a_k)+\epsilon _i)\to \Ii_{o,H}((a_1,\dots ,a_k)+\epsilon _i)\to
0
\end{equation}
Since $h^1(H,\Ii_{o,H}((a_1,\dots ,a_k)+\epsilon _i)=h^1(\Ii_Z(a_1,\dots ,a_k))=0$, \eqref{eqf1} gives $h^1(\Ii
_{Z\cup \{o\}}((a_1,\dots ,a_k)+\epsilon _i))=0$ and hence $h^0(\Ii
_{Z\cup \{o\}}((a_1,\dots ,a_k)+\epsilon _i))=h^0(\Ii_Z(a_1,\dots ,a_k))-1$.

Part (b) is proved in a similar way, but it is also a formal consequence of part (a) applied $k$ times.
\end{proof}

\begin{lemma}\label{f2}
Fix $(a_1,\dots ,a_k)\in \NN^k$ and an integer $s$ such that $0<s<\prod_{i=1}^{k} \binom{n_i+a_i}{n_i}$.
Then a general $T\in |\Ii_Z(a_1,\dots,a_k)|$ is irreducible.
\end{lemma}

\begin{proof}
First assume that $T$ has a multiple component $W$. Since the linear system $|W|$ has no base point we get a
more general element of $|\Ii_Z(a_1,\dots,a_k)|$ with a component with smaller multiplicity. Iterating this procedure we get
that a general $T$ has no multiple components. Now assume that $T=T_1\cup \cdots \cup T_e$ with all $T_j$ distinct irreducible
components and $e>1$. A monodromy argument shows that, for a general $Z$, all $T_j$ have the same multidegree and contain the
same number of elements of $Z$ and that no element of $Z$ is contained in $2$ different irreducible components of $T$. Hence
$s$ and each $a_i$ are divided by $e$. Since $\prod_{i=1}^{k}\binom{n_i+a_i}{a_i} -s > e(-s/e+\prod_{i=1}^{k}
\binom{n_i+(a_i/e)}{n_i})$, we get a contradiction. 
\end{proof}

\begin{lemma}\label{f2}
Fix $(a_1,\dots ,a_k)\in (\NN\setminus \{0\})^k$ and an integer $s$ such that $0<s<\prod_{i=1}^{k} \binom{n_i+a_i}{n_i}$. Let
$\Bb$ be the base locus $\Bb$ of $|\Ii_Z(a_1,\dots,a_k)|$. Set
$e:=
\min
\{n_1+\cdots +n_k,\prod_{i=1}^{k}\binom{n_i+a_i}{n_i} -s\}$. Then $\Bb$ has codimension at least $e$. If $\Bb$ has
codimension $e <n_1+\cdots +n_k$, then $\Bb$ is irreducible.
\end{lemma}

\begin{proof}
Since $a_i>0$ for all $i$,  $\Oo_X(a_1,\dots ,a_k)$ is very ample. Let $E$ be the intersection of $e$
general elements of $|\Oo_X(a_1,\dots ,a_k)|$. By Bertini's theorem $E$ is smooth of codimension $e$ and irreducible if
$e<n_1+\cdots +n_k$. If $e<n_1+\cdots +n_k$ take, instead of $Z$, $s$ general elements of $E$. If $e=n_1+\cdots +n_k$ take
instead of $Z$ any $s$ elements of $E$.
\end{proof}

\begin{proposition}\label{f3}
Fix $(a_1,\dots ,a_k)\in (\NN\setminus \{0\})^k$ and an integer $s$ such that $0<s<\prod_{i=1}^{k}
\binom{n_i+a_i}{n_i}-n_1-\cdots -n_k$. Then $Z$ is the scheme-theoretic base locus $\Bb$ of $|\Ii_Z(a_1,\dots,a_k)|$.
\end{proposition}

\begin{proof}
Lemma \ref{f2} shows that $\dim \Bb =0$. We first prove that each $o\in Z$ is a degree $1$ connected component of $\Bb$. As
in the proof of Proposition \ref{f2} it is sufficient to use that by Bertini's theorem the intersection of $n_1+\cdots +n_k$
general elements of $|\Ii_Z(a_1,\dots,a_k)|$ is smooth and it has dimension $0$. Thus we only need to check that $\Bb =Z$ as a
set. Let $C\subset X$ be the intersection of $n_1+\cdots +n_k-1$ general elements of $|\Oo_X(a_1,\dots ,a_k)|$. Since
$\Oo_X(a_1,\dots ,a_k)$ is very ample, $C$ is a smooth and connected curve. Let $E$ be the intersection of $C$ with a
general element of $|\Oo_X(a_1,\dots ,a_k)|$. See
$C$ embedded in a projective space $\PP^m$ by the restriction to $C$ of the linear system system $|\Oo_X(a_1,\dots ,a_k)|$.
Fix $A\subset E$ such that $\#A =s$. Note that $\#E\ge s+2$. The  generic hyperplane
section of $C\subset \PP^m$ is in uniform position in characteristic $0$ (see \cite{acgh}, pag 109), and in positive characteristic (see \cite [Th. 0.1]{R}). Thus each $s+2$ of the points of $E$ has the
same properties with respect to $|\Oo_X(a_1,\dots ,a_k)|$. Thus if
$|\Ii _A(a_,\dots ,a_k)|$ has a base locus $\Dd$ containing at least one point $o\in E\setminus A$, then it would have
$h^0(\Ii _A(a_1,\dots ,a_k)) < h^0(\Oo_X(a_1,\dots ,a_k)) -1$, a contradiction. 
\end{proof}

\begin{corollary}\label{f4}
Fix $(a_1,\dots ,a_k)\in (\NN\setminus \{0\})^k$ and an integer $s$ such that $0<s<\prod_{i=1}^{k} \binom{n_i+a_i}{n_i}$.
Fix an integer $i\in \{1,\dots,k\}$. Then $Z$ is the scheme-theoretic base locus of $|\Ii_Z((a_1,\dots,a_k)+\epsilon_i)|$.
\end{corollary}

\begin{proof}
By Proposition \ref{f3} it is sufficient to check that $h^0(\Ii_Z((a_1,\dots,a_k)+\epsilon_i))>n_1+\cdots +n_k$. Since
$s<\prod_{i=1}^{k} \binom{n_i+a_i}{n_i}$ it is sufficient to use that $(\prod_{h\ne i} \binom{n_h+a_h}{n_h}\times
\binom{n_i+a_i-1}{n_i-1} \ge \prod_{h\ne i} (n_h+1)\times \binom{n_i+a_i-1}{n_i-1}\ge n_1+\cdots +n_k$.
\end{proof}

We add an example describing why we made some assumptions 
showing that our results holds even in cases in which the multihomogeneous ideal is not generated at that level. 

\begin{example}\label{e1e1}
Fix a positive integer $s$ and $(a_1,\dots ,a_k)\in \NN$. Let $S\subset X$ be a general subset. Since $s>0$, a necessary
condition for the base point freeness of $|\Ii_S(a_1,\dots ,a_k)|\ne \emptyset$. Since $S$ is general, $|\Ii_S(a_1,\dots
,a_k)|\ne \emptyset$ if and only if $s<\prod_{i=1}^{k}\binom{n_i+a_i}{n_i}$. Set $z:= \prod_{i=1}^{k}\binom{n_i+a_i}{n_i}
-s$. If $z>0$ the generality of $S$ implies $h^1(\Ii_S(a_1,\dots ,a_k)) =0$. Thus assuming a vanishing theorem for the
multidegree for which we want to prove base point freeness is not very restrictive. However, if $a_i=0$ for some $i$, then any
$D\in |\Ii_S(a_1,\dots ,a_k)|$ vanishes on the set $\pi_i^{-1}\pi_i(S) \ne S$, because $s>0$. Thus we always assumed in this
section $(a_1,\dots ,a_k)\in (\NN\setminus \{0\})^k$. Since $S\ne \emptyset$ and $\dim S =0$, to get that $S$ is
set-theoretically intersection of elements of $|\Ii_S(a_1,\dots ,a_k)|$ we need $z\ge n_1+\dots+n_k$. Assume $z= n_1+\cdots
+n_k$. To get that $S$ is scheme-theoretically the base locus of $|\Ii_S(a_1,\dots ,a_k)|$ we need that $S$ is
(scheme-theoretically) the complete intersection of $n_1+\cdots +n_k$ elements of $|\Ii_S(a_1,\dots ,a_k)|$. This is sometimes
possible (e.g. if $s=k=2$, $n_1=n_2=1$), but almost always it is impossible, just because the self-intersection of
$n_1+\cdots +n_k$ forms of multidegree $(a_1,\dots ,a_k)$ is bigger (usually much bigger) than $n_1+\cdots +n_k$. For
instance take $k=2$, $n_1=n_2=1$ and $a_1=a_2=2$. Thus $n_1+n_2=2$, $h^0(\Oo _X(2,2)) =9$ (and hence we should take $s=7$),
but $\Oo_X(2,2)\cdot\Oo_X(2,2) =8$. 
\end{example}

\begin{remark}\label{e1e2}
Take $(a_1,\dots ,a_k)\in (\NN\setminus \{0\})^k$ and a positive integer $s$ such that $s$ is smaller than the intersection
number
of $n_1+\cdots +n_k$ copies of $\Oo_X(a_1,
\dots ,a_k)$. Fix any $S\subset X$ such that $\#S=s$ and $\Ii_S(a_1,\dots ,a_k)$ is globally generated. The assumption on $s$
implies that the image of $X$ by the rational map induced by $|\Ii_S(a_1,\dots ,a_k)|$ has image $n_1+\cdots+n_k$ and hence it
is generically finite.
\end{remark}

\end{document}